\documentclass[11pt]{article}

\usepackage{amsmath, amscd, amssymb, mathrsfs, accents, amsfonts, amsthm}
\usepackage{dirtytalk}
\usepackage{hyperref}
\usepackage{cleveref}
\usepackage{stmaryrd}
\usepackage[margin=2.5cm]{geometry}
\usepackage{ytableau}

\theoremstyle{definition}

\newtheorem*{theorem*}{Theorem}

\newtheorem{lemma}{Lemma}

\DeclareMathOperator{\sgn}{sgn}
\DeclareMathOperator{\id}{id}

\title{Lattice paths in Young diagrams}
\author{TK Waring}
\date{\today}

\begin{document}

\maketitle

\begin{abstract}
	Fill each box in a Young diagram with the number of paths from the bottom of its column to the end of its row, using steps north and east. Then, any square sub-matrix of this array starting on the south-east boundary has determinant one. We provide a --- to our knowledge --- new bijective argument for this result. Using the same ideas, we prove further identities involving these numbers, which correspond to an integral orthonormal basis of the inner product space with Gram matrix given by the array in question. This provides an explicit answer to a question (listed as unsolved\footnote{In an addendum \cite[p. 584]{stanley}, Stanley notes that Robin Chapman settled the {\em existence} problem stated in the exercise. This argument doesn't appear to be available anywhere, and in this note we provide the required object explicitly.}) raised in Exercise 6.27 c) of Stanley's Enumerative Combinatorics.
\end{abstract}

Here we consider a problem raised in Exercises 6.26 and 6.27 of \cite[p. 232]{stanley}. These problems are solved (see \cite[\S 3 Theorem 2]{carlitz}, \cite{radoux} and the solutions on \cite[p. 267]{stanley}), but here we give a concise bijective proof, using the Lindstr\"om-Gessel-Viennot lemma.

Let $D$ be a Young diagram of a partition $\lambda$, and fill each box $(i,j) \in D$ (numbering \say{matrix-wise}: down then across) with the number of paths from $(\lambda'_j,j)$ to $(i,\lambda_i)$, using steps north and east, and staying within the diagram $D$. That is, $(i,j)$ is filled with the number of paths from the lowest square in its column to the rightmost square in its row. Call this number $D_{i,j}$. For example, with $\lambda=(5,4,3,3)$:
\ytableausetup{centertableaux}
\begin{equation}\label{eq:yd-array}
\begin{ytableau}
	16 & 7 & 2 & 1 & 1 \\
	6 & 3 & 1 & 1 \\
	3 & 2 & 1 \\
	1 & 1 & 1
\end{ytableau}
\end{equation}
Then, the matrix formed by any square sub-array with a 1 in the lower right has determinant 1. The same array of integers arises in discussions of so-called ballot sequences \cite[\S 1]{carlitz}, and of Young's lattice of partitions \cite[p. 223]{stanley-fib}. For instance, from the diagram in \cref{eq:yd-array} we have:
\[ \det \begin{pmatrix} 16 & 7 & 2 \\ 6 & 3 & 1 \\ 3 & 2 & 1  \end{pmatrix} = 1. \]

This result follows directly from the following lemma, due to Lindstr\"om-Gessel-Viennot.
\begin{lemma}\label{thm:lgv}
	Let $G$ be a locally finite directed acyclic graph, and $A=\{a_1,\dots,a_n\}$ and $B=\{b_1,\dots,b_n\}$ sets of {\em source} and {\em destination} vertices, respectively. Write $e(a,b)$ for the number of paths from $a$ to $b$ in $G$, and define a matrix $M$ by $M_{i,j}=e(a_i,b_j)$. Then,
	\[ \det M = \sum_{P = P_1,\dots,P_n} \sgn (\sigma_P), \]
	where the sum is over the collection of $n$-tuples of vertex-disjoint paths $(P_1,\dots,P_n)$ in $G$, where $\sigma_P$ is a permutation of $[n]$, and $P_i$ is a path from $a_i$ to $b_{\sigma_P(i)}$.
\end{lemma}
\begin{proof} See \cite[Chapter 29]{aigner}. \end{proof}

To see this, let $G$ be the graph with the boxes of $D$ as vertices, and directed edges from each box to its northern and eastern neighbours. Given a square $n\times n$ sub-array as above, let $a_1,\dots,a_n$ be the \say{feet} of the columns of $D$ corresponding to the columns of $M$, and $b_1,\dots,b_n$ the ends of the rows. Any path system $P$ as above must have $\sigma_P = \id$, as a pair of paths $a_i \to b_j$ and $a_j \to b_j$ must share a vertex. Moreover, there is exactly one vertex-disjoint tuple $P$ of paths with $\sigma_p = \id$. The 1 in the lower right of $M$ forces the path $a_n \to b_n$ to be a \say{hook} up then right. This implies the same of the path $a_{n-1} \to b_{n-1}$ and so forth. The unique collection of paths in our running example is (poorly) rendered in \cref{eq:yd-paths}.
\begin{equation}\label{eq:yd-paths}
\begin{ytableau}
	\looparrowright & \rightarrow & \rightarrow & \rightarrow & \rightarrow \\
	\uparrow & \looparrowright & \rightarrow & \rightarrow \\
	\uparrow & \uparrow & \looparrowright \\
	\uparrow & \uparrow & \uparrow
\end{ytableau}
\end{equation}

Exercise 6.27 offers an extension, which is also resolved by our method. Suppose that $D$ is self-conjugate (ie $\lambda=\lambda'$), and let $n$ be the size of the Durfee square of the diagram $D$ --- that is, the largest $n$ such that $\lambda_n \ge n$. Let $x_1,\dots,x_n$ be a basis for a real vector space $V$, and define an inner product on $V$ by
\[ \langle x_i, x_j \rangle = D_{i,j}. \]

We exhibit an integral orthonormal basis for $V$. If $G_k = \det [D_{i,j}]_{k \le i,j \le n}$ is the \say{Gram determinant}, then, using Cramer's rule, the result of applying the Gram-Schmidt process to the vectors $x_n,x_{n-1},\dots$ (in that order) is a basis $y_n,\dots,y_1$ of $V$ given by:
\[ G_{j-1}\cdot y_j = \det \begin{pmatrix} x_j & \langle j,j+1 \rangle & \dots & \langle j,n \rangle \\ x_{j+1} & \langle j+1,j+1 \rangle & \dots & \langle j+1,n \rangle \\ \vdots & \vdots & & \vdots \\ x_{n} & \langle n,j+1 \rangle & \dots & \langle n,n \rangle \end{pmatrix} = \det \begin{pmatrix} x_j & D_{j,j+1} & \dots & D_{j,n} \\ x_{j+1} & D_{j+1,j+1} & \dots & D_{j+1,n} \\ \vdots & \vdots & & \vdots \\ x_{n} & D_{n,j+1} & \dots & D_{n,n} \end{pmatrix}  \]
Observe that the matrix in the formal determinant given here is the $(n-j+1)\times (n-j+1)$ submatrix of the Durfee square of $D$, with the first column replaced by $x_j,\dots,x_n$. As such, the above result implies that the Gram determinant $G_{j-1} = 1$, and as such the basis $y_1,\dots,y_n$ is integral. The norm of $y_j$ is $G_j / G_{j-1} = 1$.

Using the above interpretation of determinants in terms of lattice paths, we can derive the coefficients explicitly. Expanding our expression by cofactors, we obtain an expression of the form $y_j = \sum_{i=j}^n (-1)^{i-j} c_{ij}x_i$, with coefficients
\[ c_{ij} = \det \begin{pmatrix} D_{j,j+1} & \dots & D_{j,n} \\ \vdots & & \vdots \\ \widehat{D_{i,j+1}} & \dots & \widehat{D_{i,n}} \\ \vdots & & \vdots \\ D_{n,j+1} & \dots & D_{n,n} \end{pmatrix}, \]
where the hat denotes omitting that row. This is the path matrix from $a_{j+1},\dots,a_n$ to $b_{j},\dots,\widehat{b_i},\dots,b_n$. For example, with $j=1$ and $i=2$, using the tableau given above, $c_{ij}$ is the path determinant of:
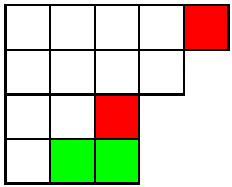
\[ \ydiagram [*(white)]{4,4,2,1} * [*(red)]{4+1,0,2+1} * [*(green)]{0,0,0,1+2} \]
First observe that, again, we can restrict ourselves to tuples $(P_j,\dots,P_{n})$ with $\sigma_P=\id$, for the same reason as above. Secondly, for any $k>i$, the path $P_k$ from $a_{k} \to b_{k}$ is uniquely determined (indeed, it is the same \say{hook} described in the original problem). For each $k<i$, the path $P_k: a_{k+1}\to b_{k}$ is determined by a number $m_k$ so that it has the form:
\[ (\lambda'_{k+1},k+1),\dots,(k+1,k+1),\dots,(k+1,m_k),(k,m_k),\dots,(k,\lambda_{k}),\]
where $k+1 \le m_k \le \lambda_k$ In the above example, we have $m_k \in \{2,3,4\}$, corresponding to the paths:
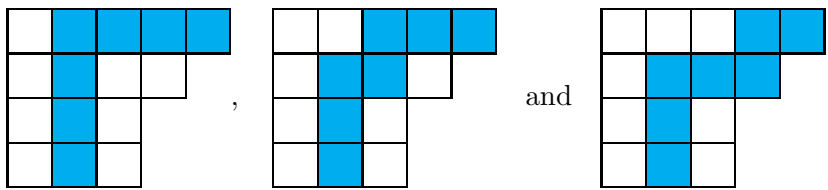
\[ \ydiagram [*(cyan)]{1+4,1+1,1+1,1+1} * {5,4,3,3}, \quad  \ydiagram [*(cyan)]{2+3,1+2,1+1,1+1} * {5,4,3,3} \quad\text{and}\quad \ydiagram [*(cyan)]{3+2,1+3,1+1,1+1} * {5,4,3,3} \]
Since $P_k$ cannot intersect $P_{k+1}$, we have $m_k < m_{k+1}$, and to avoid going outside the Young diagram, we must have $m_k \le \lambda_{k+1}$. In fact, since $\lambda_k \ge \lambda_{k+1}$, applying the second requirement to $m_{i-1}$ is sufficient. Therefore, the sequence $m_j,\dots,m_{i-1}$ is uniquely determined by an $(i-j)$-subset of $\{ j+1, \dots, \lambda_{i} \}$. Since any such a sequence determines a unique tuple $P_j,\dots,P_n$, we have:
\[ c_{ij} = \binom{\lambda_i - j}{i-j}.\]
Explicitly, this give us the expansion:
\[ y_j = \sum_{i=j}^n (-1)^{i-j} \binom{\lambda_i - j}{i-j}x_i. \]

In this example, we glossed over the requirement that $\lambda$ be self-conjugate, which allows for the interpretation of the above as an inner product. The argument goes through regardless, demonstrating the following identity for $i\ge j$:
\begin{equation}\label{eq:main-id} \langle y_j,x_i \rangle = \sum_{k=j}^n (-1)^{j-k} D_{ki} \binom{\lambda_k - j}{k - j} = \delta_{ij}.  \end{equation}
Applied to the conjugate, we have:
\[ \langle y_j',x_i \rangle = \sum_{k=j}^n (-1)^{j-k} D_{ik} \binom{\lambda_k' - j}{k - j} = \delta_{ij}.  \]
Combined, these identities determine the values $D_{ij}$ for $1 \le i,j \le n$. Cutting off initial rows or columns from the Young diagram $D$, the values of $D_{ij}$ outside the Durfee square could also be computed.

This result reduces to, and provides a bijective proof of, the special cases of Exercise 6.27 a) and b). If $\lambda = (2n+1,2n,\dots,2,1)$ then $D_{ij} = C_{2n+2-i-j}$, and the orthonormal basis $y_j$ is:
\[ y_j = \sum_{i=j}^{n+1} (-1)^{i-j}\binom{2n+2 - i - j}{i-j}x_i, \]
If we let primes denote the reflection $i' = (n+1) - i$, we get $\langle x_{i'},x_{j'} \rangle = C_{i'+j'}$ and,
\[ y_{j'} = \sum_{i'=0}^{j'} (-1)^{j'-i'}\binom{i'+j'}{j'-i'}x_{i'}, \]
as expected.

As a final example, if $\lambda = (n,n,\dots,n)$ is the partition of $n^2$, then $D_{ij} = \binom{2n - i - j}{n-i}$, and $c_{ij} = \binom{n-j}{i-j}$. The identity in question is:
\[ \langle y_j,x_i \rangle = \sum_{k=j}^n (-1)^{j-k} \binom{2n - i - k}{n-i} \binom{n - j}{k - j} = \delta_{ij}.  \]
Making the substitution $m = n-m$ on the indexes $i,j,k$, and extending the sum with terms $= 0$, this the sum is:
\begin{equation}\label{eq:subs} \sum_{k} (-1)^{k-j} \binom{i+k}{i} \binom{j}{j-k} = \binom{i}{i-j},  \end{equation}
where we have used the following, which is \cite[\S 1.2.6 eq. 23]{knuth}:
\[ \sum_{k} (-1)^{r-k}\binom{r}{k}\binom{s+k}{n} = \binom{s}{n-r}. \]
Since we require that $j \ge i$ (opposite to \cref{eq:main-id} after the substitution made in \cref{eq:subs}), this implies the claimed identity.

\clearpage
\bibliographystyle{myalpha}
\bibliography{young-refs}

\end{document}